\documentclass[reqno]{amsart}
\usepackage{amssymb}
\numberwithin{equation}{section}
\newtheorem{theorem}{Theorem}[section]
\newtheorem{proposition}[theorem]{Proposition}

\theoremstyle{definition}

\theoremstyle{definition} %%{remark}

\newcommand{\bea}{\begin{eqnarray}}
\newcommand{\eea}{\end{eqnarray}}
\newcommand{\beas}{\begin{eqnarray*}}
\newcommand{\eeas}{\end{eqnarray*}}
\newcommand{\beq}{\begin{equation}}
\newcommand{\eeq}{\end{equation}}

\newcommand{\cR}{\mathcal R}
\newcommand{\cL}{\mathcal L}
\newcommand{\cM}{\mathcal M}

\newcommand{\cPM}{\mathcal{P\!M}}

\newcommand{\C}{\mathbb C}
\newcommand{\R}{\mathbb R}
\newcommand{\N}{\mathbb N}

\newcommand{\EE}{\mathbb E}
\newcommand{\FF}{\mathbb F}

\newcommand{\ZZ}{\mathbb Z}

\newcommand{\Ver}{|\!|}

\newcommand{\ve}{\varepsilon}

\DeclareMathSymbol{\complement}{\mathord}{AMSa}{"7B}

\def\vv<#1>{\langle #1\rangle}
\def\Vv<#1>{\bigl\langle #1\bigr\rangle}

\begin{document}

% TOPMATTER

\title[Rayleigh-Taylor Instability]
{On the Rayleigh-Taylor Instability for the two-Phase Navier-Stokes Equations}

\author[J.~Pr\"uss]{Jan Pr\"uss}
\address{Institut f\"ur Mathematik \\
         Martin-Luther-Universit\"at Halle-Witten\-berg\\
         D-06120 Halle, Germany}
\email{jan.pruess@mathematik.uni-halle.de}

\author[G.~Simonett]{Gieri Simonett}
\address{Department of Mathematics\\
         Vanderbilt University \\
         Nashville, TN~37240, USA}
\email{simonett@math.vanderbilt.edu}

 \date{August 22, 2009}
\thanks{The research of the second author was partially
supported by the NSF Grant DMS-0600870.}


\begin{abstract}
The two-phase free boundary problem with surface tension and downforce 
gravity for the Navier-Stokes system is considered in
a situation where the initial interface is close to equilibrium. The boundary symbol of this problem admits zeros in the unstable halfplane in case the heavy fluid is on top of the 
light one, which leads to the well-known Rayleigh-Taylor instability. Instability is proved rigorously in an $L_p$-setting by means of an abstract instability result due to Henry \cite{henry}.
\end{abstract}
%%%%%%%%%%%%%%%%%%%%%%%%%%%%%%%%%%%%%%%%
%%%%%%%%%%%%%%%%%%%%%%%%%%%%%%%%%%%%%%%%
\maketitle

%%%%%%%%%%%%%%%%%%%%%%%%%%%%%%%%%%%%%%%%%%%%%
\section{Introduction}
%%%%%%%%%%%%%%%%%%%%%%%%%%%%%%%%%%%%%%%%%%%%%
Of concern is the
motion of two immiscible, viscous, incompressible capillary fluids,
{\em fluid$_1$} and {\em fluid$_2$}, 
% with constant densities  $\rho_i$ and (dynamic) viscosities $\mu_i$,
that occupy the regions
\begin{equation*}
\Omega_i(t)=\{(x,y)\in\R^n\times\R: (-1)^i(y-h(t,x))>0,\ t\ge 0\},
\quad i=1,2.
\end{equation*}
The fluids are separated by a sharp interface \
\begin{equation*}
\Gamma(t)=\{(x,y)\in\R^n\times\R: y=h(t,x),\ t\ge 0\}
\end{equation*}
with an unknown function $h$
that needs to be determined as part of the problem.
The motion of the fluids is governed by the incompressible
Navier-Stokes equations with surface tension and downforce gravity
and reads as follows, where $i=1,2$;
%%%%%%%%%%%%%%%%
\begin{equation}
\label{NS-2phase}
\left\{
\begin{aligned}
\rho_i\big(\partial_tu+(u\cdot\nabla)u\big)-\mu_i\Delta u+\nabla q
                       & = -\rho_i \gamma_a e_{n+1} &\ \hbox{in}\quad &\Omega_i(t)\\
    {\rm div}\,u & = 0 &\ \hbox{in}\quad &\Omega_i(t) \\
             %  u=\ab 0 &\ \hbox{on}&\partial\Omega \\
           -{[\![S(u,q)\nu]\!]}& = \sigma\kappa \nu &\ \hbox{on}\quad &\Gamma(t)\\
            {[\![u]\!]}& = 0 &\ \hbox{on}\quad &\Gamma(t) \\
                      V& = u\cdot \nu&\ \hbox{on}\quad &\Gamma(t)\\
             u(0)& = u_0&\ \hbox{in}\quad &\Omega_i(0)\\
              \Gamma(0)& = \Gamma_0\,.\\
\end{aligned}
\right.
\end{equation}
%%%%%%%%%%%%%%%%%%%%%%%%
The constants $\rho_i>0$ and $\mu_i>0$ denote the densities
and the viscosities of the respective fluids,
$\sigma$ stands for the surface tension and 
$\gamma_a$ is the acceleration of gravity.
Moreover,
$S(u,q)$ is the stress tensor defined by
\begin{equation*}
S(u, q)=\mu_i\big(\nabla u+(\nabla u)^{\sf T}\big)-qI
\quad \text{in}\quad \Omega_i(t),
\end{equation*}
and
$
[\![v]\!]=(v_{|_{\Omega_2(t)}}-v_{|_{\Omega_1(t)}}\big)|_{\Gamma(t)}
$
denotes the jump of the quantity $v$, defined on the respective
domains $\Omega_i(t)$, across the interface $\Gamma(t)$.
Finally,
$\kappa=\kappa(t,\cdot)$ is the mean curvature of the free boundary $\Gamma(t)$,
$\nu=\nu(t,\cdot)$ is the unit normal field on
$\Gamma(t)$, and $V=V(t,\cdot)$ is the normal velocity of $\Gamma(t)$.
Here we use the convention that
$\nu(t,\cdot)$ points from $\Omega_1(t)$ into $\Omega_2(t)$, and
that $\kappa(x,t)$ is negative when $\Omega_1(t)$ is convex
in a neighborhood of $x\in\Gamma(t)$.
System \eqref{NS-2phase} comprises the {\em two-phase 
Navier-Stokes equations with surface tension subject to gravity}.
%%%%%%%%%%%%%%%%%
In order to economize our notation, we set
\begin{equation*}
\rho=\rho_1\chi_{\Omega_1(t)}+ \rho_2\chi_{\Omega_2(t)},
\quad
\mu=\mu_1\chi_{\Omega_1(t)}+\mu_2\chi_{\Omega_2(t)},
\end{equation*}
where $\chi$ denotes the indicator function. It is convenient to introduce the {\em modified pressure} $\tilde{q}:= q+\rho\gamma_a y$.
With this convention system \eqref{NS-2phase} can be recast as
%%%%%%%%%%%%%%%%
\begin{equation}
\label{NS-2}
\left\{
\begin{aligned}
 \rho\big(\partial_tu+(u\cdot\nabla)u\big)-\mu\Delta u+\nabla \tilde{q}
                   &=  0\ &\hbox{in}\quad &\Omega(t)\\
      {\rm div}\,u & = 0\ &\hbox{in}\quad &\Omega(t) \\
     -{[\![S(u,\tilde q)\nu]\!]}
        & = \sigma\kappa \nu +[\![\rho]\!]\gamma_a y&\hbox{on}\quad &\Gamma(t)\\
       {[\![u]\!]} & = 0\ &\ \hbox{on}\quad &\Gamma(t) \\
                  V& = u\cdot\nu &\ \hbox{on}\quad &\Gamma(t)\\
               u(0)& = u_0&\ \hbox{in}\quad &\Omega_0\\
          \Gamma(0)& = \Gamma_0\,.\\
\end{aligned}
\right.
\end{equation}
%%%%%%%%%%%%%%%%%%%

Given are the initial velocity
$u_0:\Omega_0\to \R^{n+1}$ with $\Omega_0:=\Omega_1(0)\cup\Omega_2(0)$
as well as the initial position $\Gamma_{0}=\text{graph}\,(h_0)$.
The unknowns are the velocity field
$u(t,\cdot):\Omega(t)\to \R^{n+1}$,
the pressure field $q(t,\cdot):\Omega(t)\to \R$,
and the free boundary $\Gamma(t)$,
where $\Omega(t):=\Omega_1(t)\cup\Omega_2(t)$.
\smallskip\\
%%%%%%%%%%%%%%%%%%%%%%%%%%%%%
In case that $\Omega_1(t)$ is a bounded domain, $\gamma_a=0$, and
$\Omega_2(t)=\emptyset$, one obtains
the {\em one-phase} Navier-Stokes equations with surface tension,
describing the motion of an isolated volume of fluid.
For an overview of the existing literature in this case we refer
to the recent publications \cite{PrSi09a, SS07, SS09, So03b}.
\smallskip\\
The motion of a layer of  viscous,
incompressible fluid in an ocean of infinite extent,
bounded below by a solid surface and above by a free surface which includes
the effects of surface tension and gravity
(in which case $\Omega_0$ is a strip, bounded above
by $\Gamma_0$ and below by a fixed surface $\Gamma_b$)
has been considered by~\cite{Al87, Bea84, BeaNi84, SS09, Ta96, TT95}.
If the initial state and the initial velocity are close
to equilibrium, global existence of solutions is proved
in \cite{Bea84} for $\sigma>0$, and in \cite{TT95} for $\sigma\ge 0$,
and the asymptotic decay rate for $t\to\infty$
is studied in~\cite{BeaNi84}.
We also refer to \cite{BPS05}, where in addition the
presence of a surfactant on the free boundary and in one of the bulk phases is 
considered but gravity is neglected.
\smallskip\\
%%%%%%%%%%%%%%%%%%%%
Previous results concerning the {\em two-phase problem}~\eqref{NS-2}
with $\gamma_a=0$ in the $3D$-case
are obtained in \cite{Deni91, Deni94, DS95,Tanaka93}.
In more detail, Densiova~\cite{Deni94}
establishes existence and uniqueness of solutions
(of the transformed problem in Lagrangian coordinates)
with $v\in W^{s,s/2}_2$ for $s\in (5/2,3)$
in case that one of the domains is bounded.
Tanaka \cite{Tanaka93} considers the two-phase Navier-Stokes equations
with thermo-capillary convection in bounded domains, and he obtains
existence and uniqueness of solutions
with $(v,\theta)\in W^{s,s/2}_2$ for $s\in (7/2,4)$, with $\theta$ denoting the temperature.
\smallskip\\
%%%%%%%%%%%%%%%%%%%%%%%%
Here we are interested in the situation where $\Gamma_0$ is close to
a plane, say $\R^n$ with $n\ge 2$, i.e.\ $\Gamma_0$ is a graph over $\R^n$, given by a function $h_0$ that is small in an appropriate norm. Then it is
natural to transform the problem to a flat fixed interface,
and solve the resulting quasilinear evolution problem.
Our basic well-posedness and regularity result for
problem \eqref{NS-2} reads as follows.
%%%%%%%%%%%%%%%%%%%%%%%%%
\begin{theorem}
\label{th:1.1} 
Fix $p>n+3$ and let 
$$
(u_0, h_0)\in W^{2-2/p}_p(\Omega_0,\R^{n+1})\times W^{3-2/p}_p(\R^n)
$$
be given.
Assume that the compatibility conditions
\begin{equation}
\label{compcond}
{\rm div}\, u_0=0\; \mbox{ on }\; \Omega_0,\quad
[\![\mu D_0\nu_0-\mu(\nu_0\cdot D_0\nu_0)\nu_0]\!]=0,\quad [\![u_0]\!]=0\;
\mbox{on }\; \Gamma_0,
\end{equation}
are satisfied, where $D_0=\big(\nabla u_0+(\nabla u_0)^{\sf T}\big)$,
and $\nu_0$ is the unit normal field on $\Gamma_0$.
Then for each $a>0$ there exists $\eta>0$ such that for 
$$
\Ver u_0\Ver_{W^{2-2/p}_p(\Omega_0)} +\Ver h_0\Ver_{W^{3-2/p}_p(\R^n)}<\eta,
$$ 
there exists a unique classical solution $(u,\tilde{q},\Gamma)$
of problem \eqref{NS-2} on $(0,a)$. In addition,
$\displaystyle \cM=\cup_{t\in(0,a)}\big(\{t\}\times\Gamma(t)\big)$
is a real analytic manifold, 
and the function $(u,\tilde q):\mathcal O\rightarrow\R^{n+2}$ is real analytic,
where
$\displaystyle\mathcal O:=\cup_{t\in (0,a)}\big(\{t\}\times \Omega(t)\big).$
%	$$\mathcal O:=\{(t,x,y):\; t\in(0,t_0),\; x\in\R^n,
%	y\neq h(t,x)\},$$
\end{theorem}
%%%%%%%%%%%%%%%%%%%%%%
\begin{proof}
This result is proved in \cite{PrSi09} in case that $\gamma_a=0$. The proof given there extends  to the case $\gamma_a>0$, as the additonal term $[\![\rho]\!]\gamma_a h$ on the interface is of lower order. Actually, in Section 6 we shall give a different existence proof based on the implicit function theorem.
\end{proof}
%%%%%%%%%%%%%%%%%%%%
\noindent
We mention that system \eqref{NS-2} has also been analyzed in \cite{PrSi09a}
for initial data that are not necessarily close to equilibrium.
More precisely, it is proved in \cite{PrSi09a} that
\eqref{NS-2} admits unique solutions (on a possibly small time interval)
that have the same regularity properties as above, provided 
$\Ver \nabla h_0\Ver_\infty$ is small enough. 
%As is to be expected, the proof is considerably more involved in this case.
\medskip\\
%%%%%%%%%%%%%%%%%%%%
It is the purpose of this paper to prove mathematically rigorously that the 
trivial solution $(u,h)=(0,0)$ of problem~\eqref{NS-2} is unstable in the phase manifold
$\cPM$, to be defined below, in an $L_p$-setting 
in case that the heavy fluid overlies the lighter one, i.e.\ if $\rho_2>\rho_1$.
This is the {\em Rayleigh-Taylor instability} which is well-known in 
Physics and Hydrodynamics,  cf.\ \cite{BePe54, Chan81, ElAs97, MMSZD78} 
and the references given there. 
The Rayleigh-Taylor instability manifests itself in the way
that any disturbance of the equilibrium solution $(u,h)=(0,0)$ will grow to produce spikes of the heavy fluid moving downward and bubbles of the light fluid moving upward.
The precise statement of our main result is as follows.
\goodbreak
%%%%%%%%%%%%%%%%%%%%%%
\begin{theorem}
\label{RT}
Let $p>n+3$.
Suppose that $\rho_1,\rho_2,\mu_1,\mu_2,\sigma,\gamma_a>0$ are constants
and $\rho_2>\rho_1$.
 Then the trivial equilibrium  $(u,\tilde q,h)=(0,0,0)$ is $L_p$-unstable. More precisely, there is a constant $\ve_0>0$ such that for each $\delta>0$ there are initial values 
$$
(u_0, h_0)\in W^{2-2/p}_p(\Omega_0,\R^{n+1})\times W^{3-2/p}_p(\R^n),
$$
subject to the compatibility conditions \eqref{compcond} in Theorem 1.1 with
$$\Ver u_0\Ver_{W^{2-2/p}_p}+\Ver h_0\Ver_{W^{3-2/p}_p}\leq \delta$$
such that the solution $(u,h)$ for some $t_0\in (0,a]$ satisfies
$$\Ver u(t_0)\Ver_{W^{2-2/p}_p}+\Ver h(t_0)\Ver_{W^{3-2/p}_p}\geq \ve_0.$$
\end{theorem}
%%%%%%%%%%%%%%%%%%%%%%%%%%%
\noindent
Our method depends on the proof of Theorem 1.1 presented in \cite{PrSi09}, as well as on an abstract instability
result for iterates of a mapping due to Henry \cite{henry},  applied here to the Poincar\'e map or time-one-map of the system.
To verify the assumptions in Henry's result we show that the boundary symbol $s(\lambda,\tau)$ admits zeros $(\lambda_0,\tau_0)$ in the unstable half-plane in case $\rho_2>\rho_1$ and prove that such a zero induces the spectral values $\lambda_0$ for the linearized operator of the problem at the trivial equilibrium.

%%%%%%%%%%%%%%%%%%%%%%%%%%%%%%%%%%%%%%%
\section{Reduction to a Flat Interface and Linearization}
The nonlinear problem \eqref{NS-2} can be transformed to a problem
on a fixed domain by means of the transformations
\begin{equation*}
\begin{split}
&v(t,x,y):=(u_1,\ldots,u_n)(t,x,y+h(t,x)), \\
&w(t,x,y):=u_{n+1}(t,x,y+h(t,x)), \\
&\pi(t,x,y):=\tilde q(t,x,y+h(t,x)),
\end{split}
\end{equation*}
where $t\in J=[0,a]$, $x\in \R^n$, $y\in\R$, $y\neq0$.
With a slight abuse of notation we will
in the sequel denote the transformed velocity again
by $u$, that is, we set $u=(v,w)$.
With this notation we obtain the transformed problem
%%%%%%%%%%%%%%%%%%%%%
\begin{equation}
\label{tfbns2}
\left\{
\begin{aligned}
\rho\partial_tu -\mu\Delta u+\nabla \pi&= F(u,\pi,h)
    &\ \hbox{in}\quad &\dot\R^{n+1}\\
{\rm div}\,u&= F_d(u,h)&\ \hbox{in}\quad &\dot\R^{n+1}\\
-[\![\mu\partial_y v]\!] -[\![\mu\nabla_{x}w]\!] &=G_v(u,h)
	&\ \hbox{on}\quad &\R^n\\
-2[\![\mu\partial_y w]\!] +[\![\pi]\!] -(\sigma\Delta + [\![\rho]\!] \gamma_a) h&= G_w(u,h)
	&\ \hbox{on}\quad &\R^n\\
[\![u]\!] &=0 &\ \hbox{on}\quad &\R^n\\
\partial_th-w|_{y=0}&=H(u,h) &\ \hbox{on}\quad &\R^n\\
u(0)&=u_0 &\ \hbox{in}\quad &\dot\R^{n+1}\\
h(0)&=h_0,
\end{aligned}
\right.
\end{equation}
%%%%%%%%%%%%%%%%%%%%%%%%%%%%%%
for $t>0$, where $\dot{\R}^{n+1}:=\{(x,y)\in\R^{n}\times\R:\, y\neq0\}$.
More details on this transformation, on the nonlinear right hand sides,
can be found in  \cite{PrSi09}.
Here we should point out, however, that the definition
of $G_v$ in this paper differs from that in \cite{PrSi09} in the following way:
solving the second line of formula (2.7) in \cite{PrSi09} for $[\![\pi]\!]$
and substituting the result into the expression for $G_v$ in formula (2.8) of \cite{PrSi09}
results in 
% From the second line of \eqref{2.2} follows the relation
% \begin{equation*}
% -[\![\mu\partial_yw]\!]+[\![\pi]\!]-\sigma(\Delta h -G_\kappa(h)) - [\![\rho]\!]\gamma_a h
% =[\![\mu\partial_y w]\!]+G_w(v,w,h)+\sigma G_\kappa(h)
% \end{equation*}       
% which we then substitute into the second line of the definition of $G_v$, resulting in
\begin{equation}
\begin{split}
G_v(v,w,h):&=
-[\![\mu(\nabla_{x}v+(\nabla_{x}v)^{\sf T})]\!]\nabla h
+|\nabla h|^2[\![\mu\partial_yv]\!] \\
&\quad\;\: +\big\{[\![\mu\partial_y w]\!]- (\nabla h|\, [\![\mu\nabla_{x}w]\!])
+ |\nabla h|^2  [\![\mu\partial_yw]\!]
\big\}\nabla h.
\end{split}
\end{equation}
Thus the quantity $[\![\pi]\!]$
can be eliminated in the nonlinearity $G_v$.

%%%%%%%%%%%%%%%%%%%%%%%%%%%
The linearization of (\ref{tfbns2}) at $(u,h)=(0,0)$  leads to the linear inhomogeneous problem
%%%%%%%%%%%%%%%%%%%%%%%%%%%%%
\begin{equation}
\label{linFB}
\left\{
\begin{aligned}
\rho\partial_tu -\mu\Delta u+\nabla \pi&=f
    &\ \hbox{in}\quad &\dot\R^{n+1}\\
{\rm div}\,u&= f_d&\ \hbox{in}\quad &\dot\R^{n+1}\\
-[\![\mu\partial_y v]\!] -[\![\mu\nabla_{x}w]\!]&=g_v
	&\ \hbox{on}\quad &\R^n\\
-2[\![\mu\partial_y w]\!] +[\![\pi]\!] -(\sigma\Delta h +[\![\rho]\!] \gamma_a) h&=g_w
	&\ \hbox{on}\quad &\R^n\\
[\![u]\!] &=0 &\ \hbox{on}\quad &\R^n\\
\partial_th-w|_{y=0}&=g_h &\ \hbox{on}\quad &\R^n\\
u(0)&=u_0 &\ \hbox{in}\quad &\dot\R^{n+1}\\
h(0)&=h_0.
\end{aligned}
\right.
\end{equation}
%%%%%%%%%%%%%%%%%%%%%%%%%%%%%%
We are interested in the regularity class
\begin{equation}
\label{MR}
\begin{split}
&u\in H^{1}_p(J;L_p(\R^{n+1},\R^{n+1}))\cap L_p(J;H^2_p(\dot{\R}^{n+1},\R^{n+1})), \\
&\pi\in L_p(J;\dot{H}^1_p(\dot{\R}^{n+1})),
\end{split}
\end{equation}
where $J=[0,a]$.
In the following, $W^m_p$ denote as usual
the Sobolev spaces if $m\in\mathbb Z$. For non-integer $s$, $W^s_p$
are the Sobolev-Slobodeckii spaces,
and $H^s_p$ the Bessel-potential spaces.
 For $K\in\{H,W\}$, by $\dot{K}^s_p$ we mean the homogeneous version of $K^s_p$.
Note that $H^s_p=W^s_p$ for integer values of $s$,
but that in general these spaces are different.
We refer to \cite[Section 2]{PrSi09} for more details.

If we assume a solution in the class (\ref{MR}),
then for the right hand sides $f$ and $f_d$
we necessarily have $f\in L_p(J\times\R^{n+1},\R^{n+1})$ and
\begin{equation*}
f_d\in H^1_p(J;\dot{H}^{-1}_p(\R^{n+1}))\cap L_p(J;H^1_p(\dot{\R}^{n+1})),
\end{equation*}
since the operator ${\rm div}$ maps $L_p$ into $\dot{H}^{-1}_p$.
By trace theory we necessarily have $u_0\in W^{2-2/p}_p(\dot{\R}^{n+1},\dot\R^{n+1})$,
and the lateral trace of $u$ belongs to
$$Y_0:=W^{1-1/2p}_p(J;L_p(\R^n,\R^{n+1}))\cap L_p(J;W^{2-1/p}_p(\R^n,\R^{n+1})),$$
and that of $\partial_j u$ to
$$Y_1:=W^{1/2-1/2p}_p(J;L_p(\R^n,\R^{n+1}))\cap L_p(J;W^{1-1/p}_p(\R^n,\R^{n+1})),$$
see for instance \cite{DHP07}. 
Therefore $g_v\in Y_1$, and if in addition
$$[\![\pi]\!]\in W^{1/2-1/2p}_p(J;L_p(\R^n))\cap L_p(J;W^{1-1/p}_p(\R^n)),$$
then we also have that $g_w\in Y_1.$

Concerning the regularity of the height function $h$ we note that
 the equation for
$h$ lives in the trace space $Y_0$, hence naturally
$h$ should belong to
$$ h\in W^{2-1/2p}_p(J;L_p(\R^{n}))\cap H^1_p(J;W^{2-1/p}_p(\R^n)).$$
On the other hand, the equation for the normal component of the normal stress lives in $Y_1$,
and contains the term $\Delta h$, hence $h$ should also belong to the space
$ L_p(J;W^{3-1/p}_p(\R^n)).$
These considerations lead to the following natural space for the height function $h$
$$h\in W^{2-1/2p}_p(J;L_p(\R^{n}))\cap H^1_p(J;W^{2-1/p}_p(\R^n))
\cap L_p(J;W^{3-1/p}_p(\R^n)).$$
This then implies $g_h\in Y_0$, as well as $h_0\in W^{3-2/p}_p(\R^n)$ by trace theory.
Our next theorem states that in this setting, problem (\ref{linFB}) admits maximal regularity;
the described regularities of the data are also sufficient. In particular, the solution map
defines an isomorphism between this space of data and the solution space defined above.

%%%%%%%%%%%%%%%%%%%%%%%%%%%%%%%%%%%%%%%
\begin{theorem}
Let $1<p<\infty$ be fixed, $p\neq 3/2,3$, and assume that $\sigma$, $\gamma_a$, $\rho_i$ and $\mu_i$ are positive
constants for $i=1,2$, and set $J=[0,a]$.
Then the instationary Stokes problem with free boundary \eqref{linFB} admits a unique 
solution $(u,\pi,[\![\pi]\!],h)$ with regularity
\begin{equation*}
\label{reg}
\begin{split}
&u\in H^1_p(J;L_p(\R^{n+1},\R^{n+1}))
  \cap L_p(J;H^2_p(\dot{\R}^{n+1},\R^{n+1})), \\
& \pi\in L_p(J;\dot{H}^1_p(\dot{\R}^{n+1})), \\
&[\![\pi]\!]\in W^{1/2-1/2p}_p(J;L_p(\R^{n}))\cap L_p(J;W^{1-1/p}_p(\R^{n})),\\
& h\in W^{2-1/2p}_p(J;L_p(\R^{n}))\cap H^1_p(J;W^{2-1/p}_p(\R^n))
\cap L_p(J;W^{3-1/p}_p(\R^n))
\end{split}
\end{equation*}
if and only if the data
$(f,f_d,g,g_h,u_0,h_0)$
satisfy the following regularity and compatibility conditions: 
\goodbreak
\begin{itemize}
\item[(a)]
$f\in L_p(J;L_p(\R^{n+1},\R^{n+1}))$,
\vspace{1mm}
\item[(b)]
$f_d\in H^1_p(J; \dot{H}^{-1}_p(\R^{n+1}))\cap L_p(J; H^1_p(\dot{\R}^{n+1}))$,
\vspace{1mm}
\item[(c)]
$g=(g_v,g_w)\in W^{1/2-1/2p}_p(J;L_p(\R^{n},\R^{n+1}))
\cap L_p(J;W^{1-1/p}_p(\R^{n},\R^{n+1}))$,
\vspace{1mm}
\item[(d)]
$g_h\in W^{1-1/2p}_p(J;L_p(\R^{n}))\cap L_p(J;W^{2-1/p}_p(\R^{n}))$,
\vspace{1mm}
\item[(e)]
$u_0\in W^{2-2/p}_p(\dot{\R}^{n+1},\R^{n+1})$, $h_0\in W^{3-2/p}_p(\R^n)$,
\vspace{1mm}
\item[(f)]
${\rm div}\, u_0=f_d(0)$ in $\,\dot\R^{n+1}$ and $[\![u_0]\!]=0$
on $\,\R^n$ if $p>3/2$,
\vspace{1mm}
\item[(g)]
$-[\![\mu\partial_y v_0]\!] -[\![\mu\nabla_{x}w_0]\!] ={g_v}(0)$ on
$\,\R^n$ if $p>3$.
\end{itemize}
The solution map $[(f,f_d,g,g_h, u_0,h_0)\mapsto (u,\pi, [\![\pi]\!], h)]$ 
is continuous between the corresponding spaces.
\end{theorem}
%%%%%%%%%%%%%%%%%%%%%%%%%%%%%%%%%%%%%%%%%%%
\begin{proof}
For a detailed  proof of Theorem 2.1 in case $\gamma_a=0$ we refer to \cite{PrSi09}. The proof carries over to the case $\gamma_a>0$ since the term $[\![\rho]\!]\gamma_a h$ is of lower order. The only change occurs in the boundary symbol $s(\lambda,|\xi|)$; see Section 3.
\end{proof}
%%%%%%%%%%%%%%%%%%%%%%%%%%%%%%%%%%%%% 
We also need a corresponding result for the stationary linear problem
\begin{equation}
\label{linstatFB}
\left\{
\begin{aligned}
\rho\lambda_\ast u -\mu\Delta u+\nabla \pi&=0
    &\ \hbox{in}\quad &\dot\R^{n+1}\\
{\rm div}\,u&= f_d&\ \hbox{in}\quad &\dot\R^{n+1}\\
-[\![\mu\partial_y v]\!] -[\![\mu\nabla_{x}w]\!]&=g_v
	&\ \hbox{on}\quad &\R^n\\
-2[\![\mu\partial_y w]\!] +[\![\pi]\!] -(\sigma\Delta h + [\![\rho]\!] \gamma_a) h&=g_w
	&\ \hbox{on}\quad &\R^n\\
[\![u]\!] &=0 &\ \hbox{on}\quad &\R^n\\
\lambda_\ast h-w|_{y=0}&=g_h &\ \hbox{on}\quad &\R^n,
\end{aligned}
\right.
\end{equation}
where $\lambda_\ast>0$ is sufficiently large. It reads as follows.
%%%%%%%%%%%%%%%%%%%%%%%%%%%%%%%%%%%%
\begin{theorem}
\label{thm2.2}
Let $1<p<\infty$ be fixed, and assume that $\sigma$, $\gamma_a$, $\rho_i$ and $\mu_i$ are positive
constants for $i=1,2$, and that $\lambda_\ast>0$ is large enough.
Then the stationary Stokes problem with free boundary \eqref{linstatFB} admits a unique solution $(u,\pi, h)$ with regularity
\begin{equation}
\begin{aligned}
&u\in W^{2-2/p}_p(\dot{\R}^{n+1};\R^{n+1}),\quad &&\pi\in \dot{W}^{1-2/p}_p(\dot{\R}^{n+1}),\\
&[\![\pi]\!]\in W^{1-3/p}_p(\R^{n}),\quad && h\in W^{3-3/p}_p(\R^n)),\nonumber
\end{aligned}
\end{equation}
if and only if the data
$(f,f_d,g,g_h)$
satisfy the following regularity  conditions:
\begin{itemize}
\item[(a)]
$f_d\in  W^{1-2/p}_p(\dot{\R}^{n+1})\cap \dot{H}^{-1}_p(\R^{n+1})$,
\vspace{1mm}
\item[(b)]
$g=(g_v,g_w)\in W^{1-3/p}_p(\R^{n};\R^{n+1})$,
\vspace{1mm}
\item[(c)]
$g_h\in W^{2-3/p}_p(\R^{n})$.
\end{itemize}
The solution map $[(f_d,g,g_h)\mapsto (u,\pi, [\![\pi]\!], h)]$ 
is continuous between the corresponding spaces.
\end{theorem}
%%%%%%%%%%%%%%%%%%%
\begin{proof}
The proof will be given at the end of the next section.
\end{proof}

%%%%%%%%%%%%%%%%%%%%%%%%%%%%%%%%%%%%%%%%%%%%%%%%%%%%%%%%%%%%%%%%%%%
\section{Zeros of the Boundary Symbol}
%%%%%%%%%%%%%%%%%%%%%%%%%%%%%%%%%%%%%%%%%%%%%%%%%%%%%%%%%%%%%%%%%%%

As shown in our paper \cite{PrSi09a}, the boundary symbol of the linear problem is given by
\begin{equation}
\label{bdrysymb}
s(\lambda,|\xi|):=\lambda +\frac{\sigma|\xi|^2-[\![\rho]\!]\gamma_a}{(\rho_1+\rho_2)\lambda/|\xi|
+ 4\eta_1\eta_2/(\eta_1+\eta_2)}.
\end{equation}
Here $\lambda$ denotes the co-variable of time $t$ and $\xi$ that of the tangential space variable $x\in\R^n$, and we employed the abreviations
\begin{equation*}
\omega_j=\sqrt{\rho_j\lambda +\mu_j|\xi|^2},\quad \eta_1= \sqrt{\mu_1}\omega_1 + \mu_2 |\xi|, \quad  \eta_2= \sqrt{\mu_2}\omega_2 + \mu_1 |\xi|.
\end{equation*}
The boundary symbol $s(\lambda,|\xi|)$ has been studied in detail in the papers 
\cite{PrSi08, PrSi09} in case $\gamma_a=0$, and in
\cite{PrSi09a} for $\gamma_a>0$. 
It has been shown in \cite[Remarks 3.2(b),(c)]{PrSi09a} 
that $s(\lambda,\tau)$ does not admit zeros $(\lambda,\tau)\neq(0,0)$ with
${\rm Re} \lambda\geq0$ and $\tau\ge 0$ in case $\rho_2\leq\rho_1$, i.e. in 
the stable case. On the other hand we have the following result.
%%%%%%%%%%%%%%%%%%%%%%
\begin{proposition} 
Suppose $\rho_2>\rho_1>0$ and $\mu_1,\mu_2,\sigma,\gamma_a>0$ are constants. 
Then for each $\xi\in\R^n$ with 
$$0<|\xi|< \tau_*:=(\gamma_a[\![\rho]\!]/\sigma)^{1/2}$$ 
there is $\lambda(|\xi|)>0$ such that $s(\lambda(|\xi|),|\xi|)=0.$ Every zero
of $s(\lambda,|\xi|)$ with ${\rm Re}\lambda\geq0$ is real.
\end{proposition}
%%%%%%%%%%%%%%%%%%%%%%%%%%
\begin{proof} Note first that $s(\lambda,0)\neq 0$ unless $\lambda=0$, hence we may assume $\xi\neq0$ below.
It is convenient to the use the scaling $\zeta=\lambda/\tau^2$ where $\tau=|\xi|\in(0,\infty)$. By a slight abuse of notation we set
\begin{equation*}
\omega_j(\zeta)=\sqrt{\rho_j \zeta +\mu_j},\quad \eta_1(\zeta)= \sqrt{\mu_1}\omega_1(\zeta) + \mu_2, \quad  \eta_2(\zeta)= \sqrt{\mu_2}\omega_2(\zeta) + \mu_1,
\end{equation*}
and obtain
\begin{equation*}
s(\lambda,\tau)= \tau^2 (\zeta+ \psi(\tau)k(\zeta)),
\end{equation*}
where
\begin{equation*}
\psi(\tau)= \frac{\sigma}{(\rho_1+\rho_2)\tau}-\frac{(\rho_2-\rho_1)\gamma_a}{(\rho_1+\rho_2)\tau^3},
\end{equation*}
and
\begin{equation*}
\frac{1}{k(\zeta)} = \zeta+ \frac{4}{\rho_1+\rho_2}\frac{\eta_1(\zeta)\eta_2(\zeta)}{\eta_1(\zeta)+\eta_2(\zeta)}.
\end{equation*}
Thus $(\lambda,\tau)$ is a zero of $s$ if and only if $(\zeta,\tau)$ satisfies
$ \zeta+\psi(\tau)k(\zeta)=0$. 
It has been shown in \cite{PrSi08} that ${\rm arg}\, k(\zeta)\in(-\pi/2,0]$ 
if ${\rm arg} \, \zeta\in [0,\pi/2]$. This implies that for
 ${\rm Re}\, \zeta\geq0$, $\zeta\neq0$ we have
$\zeta+\psi(\tau)k(\zeta)\neq 0$ if either $\tau\geq\tau_*$ or $\zeta$ is non-real. Thus we need to show that for $\tau\in(0,\tau_*)$ there is a a zero $\zeta(\tau)>0$ of $ \zeta+\psi(\tau)k(\zeta)=0$. But the function $\Phi(\zeta):=\zeta/k(\zeta)$ is nonnegative and strictly increasing for $\zeta\geq0$, it is zero at $\zeta=0$ and behaves like $ \zeta^2$ as $t\to\infty$, hence $\Phi(0,\infty)\supset(0,\infty)$. On the other hand, $-\psi(0,\tau_*)= (0,\infty)$ implies that for each $\tau\in(0,\tau_*)$ there is a unique $\zeta(\tau)>0$ such that $\Phi(\zeta(\tau))=-\psi(\tau)$, i.e.\ $(\zeta(\tau),\tau)$ is a zero of the function $\zeta+\psi(\tau)k(\zeta)$, which yields the zeros $(\zeta(|\xi))|\xi|^2,\xi)$ of the boundary symbol $s(\lambda,|\xi|)$.
\end{proof}
The critical value $\tau_*=\sqrt{[\![\rho]\!]\gamma_a/\sigma}$ is known as the {\em cutoff wave number} in the literature, cp.\ e.g.\ \cite{ElAs97}.
Note that $\Phi(\zeta)\sim \zeta^2$ as $\zeta\to\infty$,
while $\Phi(\zeta)\sim \zeta/k(0)$ as $\zeta\to 0$. 
This gives the asymptotics
\begin{equation*}
\zeta(\tau)\sim \sqrt{[\![\rho]\!]\gamma_a/(\rho_1+\rho_2)}\tau^{-3/2}, 
\;  \tau\to 0;
\qquad \zeta(\tau)\sim c (\tau_*-\tau),\; \tau\to \tau_*, 
\end{equation*}
with $c= 2\sigma k(0)/(\rho_1+\rho_2)\tau_*^2$, which implies for $\lambda(\tau)=\tau^2 z(\tau)$
\begin{equation*}
\lambda(\tau)\sim \sqrt{[\![\rho]\!]\gamma_a/(\rho_1+\rho_2)}\tau^{1/2}, 
\; \tau\to 0;
\qquad \lambda(\tau)\sim \frac{\sigma}{\mu_1+\mu_2} (\tau_*-\tau),\; \tau\to \tau_*. \end{equation*}
Thus $\lambda(\tau)\to 0$ as $\tau\to 0,\tau_*$, hence the function 
$[\tau\to\lambda(\tau)]$ has a positive absolute maximum in the interval
$(0,\tau_*)$ which we denote by $\lambda_\infty>0$ in the sequel.
\bigskip\\
%%%%%%%%%%%%%%%%%%%%%%%%%%
{\sl Proof of Theorem 2.2:} 
Necessity is obtained by trace theory. To prove sufficiency,
we can use the same reductions as in \cite{PrSi09}, Sections 3-5,
with the notable difference that here we only need to consider
the stationary case with $\lambda_\ast$ a fixed parameter. 
As in the proof of \cite[Theorem 5.1]{PrSi09} it then
remains to consider the boundary symbol $s(\lambda_\ast,|\xi|)$. 
% by means of the same reductions as in \cite{PrSi09}, Sections 3-5, 
% we only need to consider the boundary symbol $s(\lambda_\ast,|\xi|)$. 
We have seen above that for $\lambda_\ast>\lambda_\infty$ the boundary symbol does not vanish, and we may estimate $s(\lambda_\ast,\tau)$ from above and below by $c_j(1+|\tau|)$
with appropriate positive constants $c_j,\ j=0,1$. This is valid
for 
$$\tau\in \Sigma_\eta=\{\zeta\in \C:\, \zeta\ne 0,\; |{\arg}\,\zeta|<\eta\}$$ 
with sufficiently small $\eta>0$. Therefore by Mikhlin's Fourier multiplier theorem 
$(1+|\xi|)/s(\lambda_\ast,|\xi|)$ defines a bounded linear operator in $H^s_p(\R^n)$ and in $W^s_p(\R^n)$ for all $s\in\R$. 
$\hfill{\square} $

%%%%%%%%%%%%%%%%%%%%%%%%%%%%%%%%%%%%%%%%%%%%%%%%%%%%%%%%%
\section{The Spectrum of the Linearization}
%%%%%%%%%%%%%%%%%%%%%%%%%%%%%%%%%%%%%%%%%%%%%%%%%%%%%%%%%

As a base space for the functional analytic setting we use
$$X_0=L_{p,\sigma}(\R^{n+1};\R^{n+1})\times W^{2-1/p}_p(\R^n),$$
where the subscript $\sigma$ means solenoidal,
and we set
$$\bar{X}_1=H^2_{p}(\dot{\R}^{n+1};\R^{n+1})
\times W^{3-1/p}_p(\R^n)\big.$$
As before we use the decomposition $u=(v,w)$.
Define a closed linear operator in $X_0$ by means of
\begin{equation}
\label{A}
A(u,h)=
( -\frac{\mu}{\rho} \Delta u +\frac{1}{\rho}\nabla \pi,-w),
\end{equation}
%
%\begin{equation*}
%A\left(\begin{array}{l}\ab u\ab \\ \ab h\ab \end{array} \right)
%=\left(\begin{array}{l}\ab -\frac{\mu}{\rho} \Delta u +\frac{1}{\rho}\nabla \pi\ab 
%\\ \ab -w\ab \end{array} \right)
%\end{equation*}
with domain $X_1:=D(A)\subset \bar{X}_1$
\begin{eqnarray*}
D(A) \hspace{-0.3cm}&&=\{(u,h)\in \bar{X}_1\cap X_0:\;
[\![u]\!]=0,\;
[\![\mu\partial_y v + \mu\nabla_x w]\!]=0 \mbox{ on } \R^n\}.
\end{eqnarray*}
The pressure field $\pi$
%\in \dot{H}^1_p(\dot{\R}^{n+1})$ 
in \eqref{A}
is determined as the solution of the transmission problem
\begin{equation}
\label{transmission}
\begin{split}
&(\frac{1}{\rho}\nabla \pi|\nabla\phi)_{L_2}= (\frac{\mu}{\rho}\Delta u|\nabla \phi)_{L_2},
\quad \phi\in W^1_{p^\prime}(\R^{n+1}),\\
&\mbox{}[\![ \pi]\!]=2[\![\mu \partial_y w]\!]+\sigma \Delta h+[\![\rho]\!]\gamma_a h
\quad \mbox{ on } \R^n.
\end{split}
\end{equation}
One should compare this operator $A$ with \cite{SS09}, where the corresponding operator in the one-phase case has been introduced and analyzed.

\noindent
Concerning the transmission problem, we set
\begin{equation*}
\tilde{\pi}:=\pi/\rho, \quad f:= \mu \Delta u/\rho, 
\quad g:=2[\![\mu \partial_y w]\!]+\sigma \Delta h+[\![\rho]\!]\gamma_a h.
\end{equation*}
 Then the solution of the transmission problem consists of two parts 
 $\tilde{\pi}=\pi_1+\pi_2$. $\pi_1$ is given by
$\pi_1:= -{\rm div}\, D_n^{-1} f$, i.e.\ $\nabla\pi_1=Rf$, where $R$ denotes the Riesz-transform with symbol $\xi\otimes\xi /|\xi|^2$. 
On the other hand, $\pi_2$ solves 
\begin{equation*}
\label{pi-2}
\Delta\pi_2=0\;\: \text{in}\;\: \dot\R^{n+1},
\quad [\![\rho\pi_2]\!]=g-[\![\rho\pi_1]\!]=:g_0
\;\;\text{and}\;\; [\![\partial_y\pi_2]\!]=0\;\; \text{on}\;\:\R^n.
\end{equation*}
The solution of the latter problem is given by 
$$
\pi_2(\cdot,y)= \frac{1}{(\rho_1+\rho_2)}\;\text{sign}\,(y) P(|y|)g_0,
$$
where 
$\{P(s):s\ge 0\}$ denotes the Poisson semigroup on $\R^n$. 
Thus $\nabla\pi\in L_p(\R^{n+1})$ 
since $g$ and $g_0$ belong to $\dot{W}^{1-1/p}_p(\R^n)$ and $f\in L_p(\R^{n+1})$,
see for instance formula (2.14) in \cite{PrSi09} for the assertion for $\pi_2$.
\medskip\\
System (\ref{linFB}) can be
rewritten as the abstract evolution equation
\begin{equation} 
\label{ACP}
\frac{d}{dt}(u,h)+A(u,h)=(f_d,g_h),\quad t>0,\quad (u(0),h(0))=(u_0,h_0),
\end{equation}
provided $(f_d,g)=(0,0).$
Since (\ref{linFB}) has maximal $L_p$-regularity,
the abstract problem (\ref{ACP}) has maximal $L_p$-regularity as well.
In particular, $-A$ generates an analytic $C_0$-semigroup in $X_0$.
Concerning the spectrum $\sigma(A)$ of $A$ we have the following result.
%%%%%%%%%%%%%%%%%%%%%
\begin{proposition} 
Suppose $\rho_2>\rho_1>0$ and $\mu_1,\mu_2,\sigma,\gamma_a>0$ are constants, and let $A$ with domain $D(A)=X_1$ be defined in $X_0$ as above. \\
Then $\lambda_0\in\sigma(-A)$ for each zero 
$(\lambda_0,|\xi_0|)\in(0,\infty)\times (0,\infty)$ of $s(\lambda,|\xi|)=0.$
In particular, $\sigma(-A)\cap\C_+=[0,\lambda_\infty]$ with  $\lambda_\infty>0$ from the previous section.
\end{proposition}
%%%%%%%%%%%%%%%%%%%%%%%
\begin{proof}
(i) The idea is to show that whenever we have a zero $(\lambda_0,|\xi_0|)$ of the boundary symbol $s$ then $\lambda_0$ is an approximate eigenvalue of $-A$, hence in $\sigma(-A)$.  So suppose  $s(\lambda_0,|\xi_0|)=0$. We define functions $h_\ve$ by means of
\begin{equation*}
h_\ve(x)=e^{i\xi_0\cdot x}\chi_\ve(x),\quad x\in\R^n,\; \ve>0,
\end{equation*}
where $\chi$ is a Schwartz-function such that its Fourier-transform $\hat\chi$ is a test function with ${\rm supp}\,(\hat\chi)\subset\bar{B}_{\R^n}(0,1)$, $\chi(0)=1$ 
and $\chi_\ve(x):=\chi(\ve x)$. This implies by means of the formula 
$\hat h_\ve(\xi)= \ve^{-n}\hat\chi((\xi-\xi_0)/\ve)$ that the support of $\hat h_\ve$ is contained in the ball $\bar{B}_{\R^n}(\xi_0,\ve)$.
In particular, for all $r\in\R$, $1<p<\infty$ and $K\in\{H,W\}$,  the operators 
$D_n:=-\Delta$, $D_n^{1/2}$ and $D_n^{-1/2}$ as well as $B_0:=\lambda_0 D_n^{-1}$ act boundedly  in 
$$
K^r_{p,c}(\R^n):=\{g\in K^r_p(\R^n):
\,{\rm supp}\,(\hat g)\subset \bar{B}_{\R^n}(\xi_0,\ve)\},
$$ 
as long as, say, $\ve\leq|\xi_0|/2$. Note that the spectra of $D_n$, $D_n^{1/2}$ and $B_0$ in $K^r_{p,c}$ are contained in $[\alpha,\beta]$ for some $0<\alpha<\beta<\infty$.
\medskip\\
%%%%%%%%%%%%%%%%%%%%%%%%%%%%%%%
(ii)
 The $L_p$-norm of $h_\ve$ is $\Ver h_\ve\Ver_p= \ve^{-n/p}\Ver\chi\Ver_p$, and moreover we have
\begin{equation*}
\Ver h_\ve\Ver_{H^k_p} \leq c\ve^{-n/p}\Ver\chi\Ver_{H^k_p},
\quad k\leq 4,\quad 0<\ve\leq|\xi_0|/2.
\end{equation*}
Hence by the interpolation inequality one sees that
\begin{equation*}
c_0\ve^{-n/p}\leq \Ver h_\ve\Ver_ {K^{r}_p}\leq  c_1\ve^{-n/p},
\quad  0<\ve\leq|\xi_0|/2,\quad r\in[0,4],
\end{equation*}
where $K\in\{H,W\}$ and $c_0,c_1$ are positive constants.

Next we solve the Stokes problem
\begin{equation}
\label{Stokes}
\left\{
\begin{aligned}
\rho\lambda_0 u -\mu\Delta u+\nabla \pi&=0
    &\ \hbox{in}\quad &\dot\R^{n+1}\\
{\rm div}\,u&= 0&\ \hbox{in}\quad &\dot\R^{n+1}\\
-[\![\mu\partial_y v]\!] -[\![\mu\nabla_{x}w]\!]&=0
	&\ \hbox{on}\quad &\R^n\\
-2[\![\mu\partial_y w]\!] +[\![\pi]\!] &=(\sigma\Delta +[\![\rho]\!] \gamma_a) h_\ve
	&\ \hbox{on}\quad &\R^n\\
[\![u]\!] &=0 &\ \hbox{on}\quad &\R^n.
\end{aligned}
\right.
\end{equation}
This is possible since $\lambda_0>0$, and we obtain for the solution $(u_\ve,\pi_\ve)$ the estimate
\begin{equation*}
\Ver u_\ve\Ver_{H^2_p(\dot{\R}^{n+1})}+ \Ver\nabla \pi_\ve\Ver_{L_p(\R^{n+1})} 
\leq C\Ver h_\ve\Ver_{W^{3-1/p}_p(\R^n)}\leq Cc_1\ve^{-n/p}.
\end{equation*}
We then have
$$\lambda_0 (u_\ve, h_\ve) +A (u_\ve,h_\ve) = (0, g_\ve),$$
where $g_\ve = \lambda_0 h_\ve
+ (\sigma D_n^{1/2}-[\![\rho]\!]\gamma_a D_n^{-1/2})
(\rho_1+\rho_2)^{-1}k(\lambda_0 D_n^{-1})h_\ve$.
\medskip\\
%%%%%%%%%%%%%%%%%%%%%%%%%%%%%%
(iii) It remains to estimate $g_\ve$.
%For this purpose we consider the Fourier transform of $g_\ve$.
%Since $s(\lambda_0,\xi_0)=0$ we have
%\begin{align*}
%\hat g_\ve(\xi) &= \lambda_0 \hat h_\ve(\xi) + (\sigma|\xi| -[\![\rho]\!]\gamma_a /|\xi|) k(\lambda_0/|\xi|^2)\hat h_\ve(\xi)\\
%&= (\phi(\xi)-\phi(\xi_0))\hat h_\ve(\xi)\\
%&=(\phi(\xi)-\phi(\xi_0))\ve^{-n}\hat \chi((\xi-\xi_0)/\ve),
%\end{align*}
%where we employed the abbreviation
%$$ \phi(\xi) =(\sigma|\xi| -[\![\rho]\!]\gamma_a /|\xi|) k(\lambda_0/|\xi|^2).$$
First we observe that
\begin{equation*}
 D_n h_\ve(x) = |\xi_0|^2 h_\ve(x) -\ve r_\ve(x),\quad r_\ve(x):=e^{i\xi_0\cdot x}[ 2 i ((\xi_0|\nabla)\chi)(\ve x) + \ve(\Delta \chi)(\ve x),
\end{equation*}
hence
$$
 D_n^{-1} h_\ve(x)= |\xi_0|^{-2} h_\ve(x) + \ve|\xi_0|^{-2} D_n^{-1}r_\ve(x),
$$
and therefore
\begin{equation*}
(\zeta-B_0)h_\ve = (\zeta-\lambda_0|\xi_0|^{-2})h_\ve -\ve|\xi_0|^{-2}B_0r_\ve,
\end{equation*}
which yields
\begin{equation*}
(\zeta-B_0)^{-1}h_\ve = (\zeta-\lambda_0|\xi_0|^{-2})^{-1}h_\ve +\ve|\xi_0|^{-2}(\zeta-\lambda_0|\xi_0|^{-2})^{-1}(\zeta-B_0)^{-1}B_0 r_\ve.
\end{equation*}
With $k_1(\zeta)=(\sigma - \zeta[\![\rho]\!]\gamma_a/\lambda_0)(\rho_1+\rho_2)^{-1}k(\zeta)$, by means of Dunford's functional calculus
in $K^r_{p,c}(\R^n)$
 this yields, with a closed contour $\Gamma$ in the open right half-plane surrounding 
 $[\alpha,\beta]$, 
\begin{align*}
k_1(B_0) h_\ve &= \frac{1}{2\pi i} \int_\Gamma k_1(\zeta) (\zeta-B_0)^{-1} h_\ve\, d\zeta \\
&= k_1(\lambda_0|\xi_0|^{-2})h_\ve  
+ \frac{\ve}{2\pi i|\xi_0|^2}\int_\Gamma k_1(\zeta) (\zeta-\lambda_0|\xi_0|^{-2})^{-1} 
(\zeta-B_0)^{-1} B_0r_\ve\, d\zeta\\
&=k_1(\lambda_0|\xi_0|^{-2})h_\ve + \ve K_1r_\ve,
\end{align*}
where the operator $K_1$ is bounded and does not depend on $\ve$. Here we used the fact that $k_1(\zeta)$ is bounded. Finally, we have in a similar way
\begin{align*}
D_n^{1/2} h_\ve &=\frac{1}{2\pi i}\int_\Gamma\sqrt{\zeta}(\zeta-D_n)^{-1}h_\ve\,d\zeta \\
&= |\xi_0| h_\ve - \frac{\ve}{2\pi i}\int_\Gamma \sqrt{\zeta}(\zeta-|\xi_0|^2)^{-1}(\zeta-D_n)^{-1}\, d\zeta\, r_\ve\\
& =|\xi_0| h_\ve + \ve K_2r_\ve,
\end{align*}
with a bounded operator $K_2$ that is independent of $\ve$.
Thus in summary we have
\begin{align*}
g_\ve & =\big(\lambda_0 + (\sigma D_n^{1/2}-[\![\rho]\!]\gamma_a D_n^{-1/2})
(\rho_1+\rho_2)^{-1}k(B_0)\big)h_\ve\\
 &= s(\lambda_0,|\xi_0|)h_\ve + \ve \big( k_1(\lambda_0/|\xi_0|^2)K_2+K_1 D_n^{1/2}\big)r_\ve\\
 & =\ve \big(k_1(\lambda_0/|\xi_0|^2) K_2+K_1 D_n^{1/2}\big)r_\ve,
\end{align*}
since by assumption $s(\lambda_0,|\xi_0|)=0.$ The operators  $K_j$ are bounded and $D_n^{1/2}$ is so on functions whose Fourier transform has compact support bounded away from 0, hence we obtain
$$\Ver g_\ve\Ver_{W^{2-1/p}_p(\R^n)}\leq  C\ve \Ver r_\ve\Ver_{W^{2-1/p}_p(\R^n)} 
\le C\ve \Ver h_\ve\Ver_{W^{4-1/p}_p(\R^n)} \leq C \ve \ve^{-n/p}.$$
Scaling the functions $h_\ve$ by the factor $\ve^{n/p}$ we have   
$\Ver (u_\ve,h_\ve)\Ver_{X_0}\geq c>0$ and
$(\lambda_0+A)(u_\ve,h_\ve) \to 0$ as $\ve\to 0$. Thus $\lambda_0$ must belong to the spectrum of $(-A)$.
\medskip\\
%%%%%%%%%%%%%%%%%%%%%%%%%%%%%%%%%%%%%
(iv) Finally, if $\lambda\in \C_+$ is not in $[0,\lambda_\infty]$, then $s(\lambda,\cdot)$ does not vanish. Therefore the boundary symbol can be inverted by means of Mikhlin's multiplier theorem and then solving the remaining Stokes problem we see that $\lambda+A$ is boundedly invertible. Thus such $\lambda$ belong to the resolvent set of $A$. The proof is complete.
\end{proof}

%%%%%%%%%%%%%%%%%%%%%%%%%%%%%%%%%%%%%%%%%%%%%%%%%%%%%%%%%
\section{Henry's Instability Theorem}
%%%%%%%%%%%%%%%%%%%%%%%%%%%%%%%%%%%%%%%%%%%%%%%%%%%%%%%%%

For the reader's convenience we provide the statement of Henry's instability theorem, 
see \cite[Theorem 5.1.5]{henry}, for a more specialized situation which is appropriate for the problem under consideration in this paper.
%%%%%%%%%%%%%%%%%%%%%%%%%
\begin{theorem} 
Let $X$ be a Banach space, $T\in C^2(B_X(0,\varrho);X)$ such that $T(0)=0$ and the spectral radius
${\rm spr}\; T^\prime(0)$ of $T^\prime(0)$ is greater than one. Then the origin is unstable in the sense that there is a constant $\ve_0>0$ such that for each $\delta>0$ there is $x_\delta\in B_X(0,\delta)$ and $N\in\N$ such that the sequence 
$x_k:=T^kx_\delta\in B_X(0,\varrho)$ is well-defined for all $k\in\{1,\ldots,N\}$ and $\Ver T^Nx_\delta\Ver\geq \ve_0$.
\end{theorem}
%%%%%%%%%%%%%%%%%%%%%%%%%%%
\noindent
As an illustration consider as in \cite{Pru03} and \cite{PrWi09} the quasi-linear evolution equation
\begin{equation}\label{qlevp}
\dot{u} +A(u)u=F(u),\;t>0,\quad u(0)=u_0,
\end{equation}
where $X_1\hookrightarrow X_0$ are densely embedded Banach spaces, 
$X_\gamma:=(X_0,X_1)_{1-1/p,p}$ is a real interpolation space of order 
$1-1/p$ and power $p\in(1,\infty)$ between $X_0$ and $X_1$.
The nonlinearities 
$(F,A):X_\gamma\to X_0\times\cL(X_1,X_0)$ are of class $C^2$. Assume that $u_*\in X_1$ is an equilibrium of \eqref{qlevp}, i.e.\ $ A(u_*)u_*=F(u_*)$. Assume further that the linearization of \eqref{qlevp} has maximal $L_p$-regularity, i.e.\ the operator $A_0$ defined by
$$A_0v:= A(u_*)v +[A^\prime(u_*)v]u_*-F^\prime(u_*)v,\quad v\in X_1,$$
is $\cR$-sectorial. Due to the results in \cite{Pru03} and \cite{PrWi09}, 
the Poincar\'e map of \eqref{qlevp},  
$[u_0\mapsto Tu_0-u_*:=u(a,u_0)-u_*],$
where $u(t,u_0)$ denotes the solution of \eqref{qlevp} with $u_0\in B_{X_\gamma}(u_*,\varrho)$, 
is well-defined and satisfies the assumptions of Henry's instability theorem in 
$X:=X_\gamma$, provided $\varrho>0$ is sufficiently small. 
The derivative $T^\prime(0)$ equals $e^{-A_0a}$, which by the spectral mapping theorem for generators of analytic $C_0$-semigroups has spectral radius greater than one 
if and only if the spectrum of the operator $-A_0$ contains points in the open right half-plane. In particular, in contrary to \cite{Pru03} no spectral gap is required. Therefore we may conclude that the equilibrium $u_*$ is unstable in the natural phase space $X_\gamma$ of the problem.
\medskip\\
In the situation of the two-phase Navier-Stokes equations with surface tension and subject to gravity, things are more involved due to the inherent nonlinearity of the compatibility conditions for the transformed problem 
\begin{equation}
\label{compatibility}
\begin{split}
{\rm div}\, u&=F_d(u,h), \\
 -[\![\mu\partial_yv]\!]-[\![\mu\nabla_x w]\!]&=G_v(u,h),\\
 -2[\![\mu\partial_yw]\!]+[\![\pi]\!]-(\sigma\Delta h + [\![\rho]\!]\gamma_a)h&=G_w(u,h),
\end{split}
\end{equation}
on the interface. This leads to a nonlinear phase manifold where the semi-flow lives on. The way out is to parameterize  this manifold. Nevertheless, the linearization of the time-one map will turn out to be the operator $e^{-A}$ with $A$ described in the previous section, which by Proposition 4.1 has spectral values in the open right half-plane. This way we will still be able to apply Henry's result, and as a consequence to obtain a rigorous proof 
for the Rayleigh-Taylor instability.

%%%%%%%%%%%%%%%%%%%%%%%%%%%%%%%%%%%%%%%%%%%%%%%%%%%%%%%%%
\section{Proof of the main result}
%%%%%%%%%%%%%%%%%%%%%%%%%%%%%%%%%%%%%%%%%%%%%%%%%%%%%%%%%
\noindent
(i) In a first step we parameterize the phase manifold $\cPM$ of 
system \eqref{tfbns2} which lies in 
\begin{equation*}
X_\gamma:= \{(u,h)\in W^{2-2/p}_p(\dot{\R}^{n+1};
\R^{n+1})\times W^{3-2/p}_p(\R^n):
[\![u]\!]=0\}
\end{equation*}
and is defined by
\begin{align*} 
\cPM :=\big\{ (u,h)\in X_\gamma:\; 
{\rm div}\, u = u\cdot\nabla h \mbox{ in }\dot{\R}^{n+1}, \;\;
-[\![\mu\partial_y v]\!]-[\![\mu\nabla_x w]\!]= G_v(u,h) \big\}.
\end{align*}
We want to parameterize $\cPM$ locally near $(0,0)$ over the closed linear subspace 
$X^0_\gamma\subset X_\gamma$, defined by
\begin{equation*}
X_\gamma^0:=
\big\{ (u,h)\in X_\gamma:\; {\rm div}\, u = 0 \mbox{ in } \dot{\R}^{n+1},\;
[\![\mu\partial_y v]\!]+[\![\mu\nabla_x w]\!]= 0 \mbox{ on } \R^n\big\},
\end{equation*}
by means of an analytic map 
$\Phi:  B_{X^0_\gamma}(0,r)\to \cPM$ which is bijective onto its range such that $\Phi^\prime(0)=I$.
For this purpose we consider the problem
\begin{equation*}
\label{nonlinstokes1}
\left\{
\begin{aligned}
\rho\lambda_\ast u -\mu\Delta u+\nabla \pi&=0
    &\ \hbox{in}\ &\dot\R^{n+1}\\
{\rm div}\,u&= F_d(u+\tilde{u},h+\tilde{h})&\ \hbox{in}\ &\dot\R^{n+1}\\
-[\![\mu\partial_y v]\!] -[\![\mu\nabla_{x}w]\!] &=G_v(u+\tilde{u},h+\tilde{h})
	&\ \hbox{on}\ &\R^n\\
-2[\![\mu\partial_y (w+\tilde{w})]\!] +[\![\pi]\!] &= 
(\sigma\Delta +[\![\rho]\!] \gamma_a)(h+\tilde {h})
%\\ &\quad 
+G_w(u+\tilde{u},h+\tilde{h})
	&\ \hbox{on}\ &\R^n\\
[\![u]\!] &=0 &\ \hbox{on}\ &\R^n\\
\lambda_\ast h-w|_{y=0}&=H(u+\tilde{u},h+\tilde{h}) &\ \hbox{on}\ &\R^n,
\end{aligned}
\right.
\end{equation*}
with some fixed, sufficiently large  $\lambda_\ast>0$ and given 
$(\tilde{u},\tilde{h})\in  B_{X^0_\gamma}(0,r)\subset X_\gamma^0$.
We write this equation in short hand notation as  $L_{\lambda_\ast} z=N(z+\tilde{z})$ in
%\begin{equation*}
%\bar{X}_\gamma= [W^{2-2/p}_p(\dot{\R}^{n+1};\R^{n+1})\cap H^1_p(\R^{n+1};\R^{n+1})]\times W^{3-2/p}_p(\R^n)
%\end{equation*}
\begin{equation*}
\mathbb Z_\gamma:=\big\{(u,\pi,[\![\pi]\!],h): 
(u,h)\in X_\gamma,\; \pi\in \dot W^{1-2/p}_p(\dot\R^{n+1}),\;
[\![\pi]\!]\in W^{3-2/p}_p(\R^n)\big\}
\end{equation*}
where $z=(u,\pi,[\![\pi]\!],h)$ and $\tilde z=(\tilde u,0,0,\tilde h)$.
Since the nonlinear terms are polynomial, it is not difficult to verify that $N$ 
is real analytic and  $N^\prime(0)=0$; see \cite[Proposition~6.2]{PrSi09} for a related result.
$L_{\lambda_\ast}$ is invertible by Theorem 2.2, 
and hence the implicit function theorem yields a unique solution 
$z=\frak{Z}(\tilde u,\tilde h)\in \ZZ_\gamma$ near $0$.
The mapping 
$$\frak{Z}: B_{X^0_\gamma}(0,r)\to \ZZ_\gamma,$$
with  $r$ chosen small enough, 
is real analytic and satisfies $\frak{Z}^\prime(0)=0$. 
Denoting by $P:\ZZ_\gamma \to X_\gamma$ the projection
given by $Pz:=(u,h)$ for $z=(u,\pi,[\![\pi]\!],h)$,  
we set 
\begin{equation}
\label{Phi}
 \Phi(\tilde{u},\tilde{h}):= (\tilde{u},\tilde{h}) +\phi(\tilde{u},\tilde{h})
 \quad\text{with}\quad
 \phi(\tilde u,\tilde h):=P\frak{Z}(\tilde u,\tilde h).
\end{equation}
Then $\Phi( B_{X^0_\gamma}(0,r))\subset \cPM$,
% $\text{Im}(\Phi)\subset \cPM$, 
$\Phi: B_{X^0_\gamma}(0,r)\to X_\gamma$ is real analytic, $\Phi^\prime(0)=I$, and $\Phi$ is injective.
 Hence it remains to show local surjectivity near $0$. So suppose that $(\bar{u},\bar{h})\in\cPM$ has sufficiently small norm. Solving the problem
\begin{equation*}
%\label{nonlinstokes2}
\left\{
\begin{aligned}
\rho\lambda_\ast u -\mu\Delta u+\nabla \pi&=0
    &\ \hbox{in}\quad &\dot\R^{n+1}\\
{\rm div}\,u&= F_d(\bar{u},\bar{h})&\ \hbox{in}\quad &\dot\R^{n+1}\\
-[\![\mu\partial_y v]\!] -[\![\mu\nabla_{x}w]\!] &=G_v(\bar{u},\bar{h})
	&\ \hbox{on}\quad &\R^n\\
-2[\![\mu\partial_y w]\!] +[\![\pi]\!] -(\sigma\Delta +[\![\rho]\!] \gamma_a) h&= G_w(\bar{u},\bar{h})
	&\ \hbox{on}\quad &\R^n\\
[\![u]\!] &=0 &\ \hbox{on}\quad &\R^n\\
\lambda_\ast h-w|_{y=0}&=H(\bar{u},\bar{h}) &\ \hbox{on}\quad &\R^n,
\end{aligned}
\right.
\end{equation*}
by means of  Theorem 2.2
yields a unique solution $z=(u,\pi,[\![\pi]\!],h)\in\ZZ_\gamma$.
One readily verifies that
$(\tilde{u},\tilde{h}):= (\bar{u}-u,\bar{h}-h)$ belongs to $X_\gamma^0$ and $\phi(\tilde{u},\tilde{h})=(u,h)$, showing surjectivity of $\Phi$ near $0$. In particular, 
$\cPM\subset X_\gamma$ is a real analytic manifold near $0\in X_\gamma$.
We observe that due to
\begin{equation*}
-2[\![\mu\partial_y (w+\tilde{w})]\!] +[\![\pi]\!]=  
(\sigma\Delta +[\![\rho]\!] \gamma_a)(h+\tilde {h})
+G_w(u+\tilde{u},h+\tilde{h})\ \hbox{on}\ \R^n
\end{equation*}
the last condition in \eqref{compatibility} 
is satisfied as well.
\smallskip\\
%%%%%%%%%%%%%%%%%%%%%%%%%%%%%%
(ii) Now we proceed as in the proof of  \cite[Theorem 6.3]{PrSi09}, 
employing the notation of that proof. For a given $(u_0,h_0)\in \cPM$ we construct the extension $z_*(u_0,h_0)\in \EE_1(a)$. 
The map
\begin{equation*}
[(\tilde u_0,\tilde h_0))\mapsto z_\ast(\Phi(\tilde u_0,\tilde h_0))]
: B_{X^0_\gamma}(0,r)\to\EE(a)
\end{equation*}
is real analytic.
Therefore, fixing  $a>0$ the mapping
\begin{equation}
\label{K}
\begin{split}
&K: {_0\EE}(a)\times  B_{X_\gamma^0}(0,r)\to {_0\FF(a)}, \\ 
&K(z,(\tilde u_0,\tilde h_0)):= 
L(z+z_*(\Phi(\tilde u_0,\tilde h_0)))- N(z+z_*(\Phi(\tilde u_0,\tilde h_0)))
\end{split}
\end{equation} 
is real analytic. 
%from ${_0\EE}(a)\times  B_{X_\gamma^0}(0,r)$ to ${_0\EE}(a)$. 
We have 
$K(0,0)=0$ as well as $D_zK(0,0)=L$.
Therefore the implicit function theorem yields a real analytic 
map 
$$
\Psi: B_{X_\gamma^0}(0,r)\to {_0\EE}(a)
$$
such that 
$K(\Psi(\tilde u_0,\tilde h_0),(\tilde u_0,\tilde h_0))=0$
for all $(\tilde u_0,\tilde h_0)\in B_{X_\gamma^0}(0,r)$,
with $r$ chosen sufficiently small.
Thus 
\begin{equation}
\label{solution}
z:=\Psi(\tilde u_0,\tilde h_0)+ z_*(\tilde u_0,\tilde h_0)\in\EE(a)
\end{equation}
is the unique solution of problem (2.1) and, moreover, the mapping
\begin{equation}
\label{z-analytic}
[(\tilde u_0,\tilde h_0)\mapsto z(\tilde u_0,\tilde h_0)]:
B_{X_\gamma^0}(0,r)\to {\EE}(a))
\end{equation}
is real real-analytic.
\smallskip\\
%%%%%%%%%%%%%%%%%%%%%%%%%%%%
(iii) Having obtained a unique solution $z=z(\tilde u_0,\tilde h_0)\in\EE(a)$ 
of \eqref{tfbns2}, we can employ the same arguments as in steps (vi)--(vii) in the proof of
\cite[Theorem 6.3]{PrSi09} to establish analyticity of the solution
as stated in Theorem~\ref{th:1.1}.
\goodbreak
\smallskip\noindent
%%%%%%%%%%%%%%%%%%%%%%%%%%%%%%%
(iv) Differentiating the mapping $[(\tilde u_0,\tilde h_0)\mapsto z(\tilde u_0,\tilde h_0)]$ 
w.r.t.\ the initial value $(\tilde u_0,\tilde h_0)$ 
one sees that the linearization at $(0,0)$ of \eqref{tfbns2} is given by the Cauchy problem 
\begin{equation*}
\frac{d}{dt}(u,h)+A(u,h)=0,\quad (u(0),h(0))=(u_0,h_0).
\end{equation*}
This implies that the linearization at $(0,0)$ of the time-one-map 
$
[(\tilde u_0,\tilde h_0)]\mapsto z(1)]$ of the nonlinear problem \eqref{tfbns2} is given by $e^{-A}$. Since $(-A)$ generates an analytic $C_0$-semigroup in $X_0$ 
the spectral mapping theorem 
%for analytic semigroups 
yields $\text{spr}\, (e^{-A}) = e^{\lambda_\infty }>1$ by Proposition 4.1, hence we may apply Theorem 5.1 to obtain instability of the trivial solution. The proof of our main result is therefore complete.
$\hfill\square$
\bigskip
%%%%%%%%%%%%%%%%%%%%%%%%%%%%%%%%

\end{document}